\documentclass[final]{siamltex} 

\usepackage{graphicx}
\usepackage{amsmath,amssymb}
\usepackage{booktabs}
\usepackage{siunitx}
\usepackage{enumerate}
\usepackage{color}

\newcommand{\species}[1]{\mathrm{#1}}

\newcommand{\coloneqq}{\mathrel{\mathop:}=}

\renewcommand{\d}{\mathrm{d}}
\newcommand{\f}{\mathrm{f}}
\newcommand{\D}{\mathrm{D}}
\newcommand{\T}{\mathrm{T}}

\newcommand{\RR}{\mathbb{R}}

\newcommand{\AAA}{\mathcal{A}}
\newcommand{\CCC}{\mathcal{C}}
\newcommand{\III}{\mathcal{I}}
\newcommand{\LLL}{\mathcal{L}}

\newcommand{\lin}{\mathrm{lin}}

\title{A continuation method for the efficient solution of parametric
optimization problems in kinetic model reduction\thanks{This work was supported
by the German Research Foundation (DFG) through the Collaborative Research
Center (SFB) 568 and by the Baden-W\"urttemberg Stiftung via project HPC~11.}}

\author{Dirk Lebiedz\footnotemark[2]
   \and Jochen Siehr\footnotemark[2] \footnotemark[3]}

\begin{document} 
\maketitle

\renewcommand{\thefootnote}{\fnsymbol{footnote}}
\footnotetext[2]{Institute for Numerical Mathematics, University of Ulm,
  Helmholtzstra\ss{}e 20, 89081 Ulm, Germany
  (\texttt{dirk.lebiedz@uni-ulm.de}, \texttt{jochen.siehr@uni-ulm.de}).}
\footnotetext[3]{Interdisciplinary Center for Scientific Computing (IWR),
  University of Heidelberg, Im Neuenheimer Feld 368, 69120 Heidelberg,
  Germany.}
\renewcommand{\thefootnote}{\arabic{footnote}}

\begin{abstract} 
Model reduction methods often aim at an identification of slow invariant
manifolds in the state space of dynamical systems modeled by
ordinary differential equations. We present a predictor corrector method
for a fast solution of an optimization problem the solution of which is
supposed to approximate points on slow invariant manifolds. The corrector
method is either an interior point method or a generalized Gauss--Newton
method. The predictor is an Euler prediction based on the parameter
sensitivities of the optimization problem. The benefit of a step size
strategy in the predictor corrector scheme is shown for an example.
\end{abstract}



\pagestyle{myheadings}
\thispagestyle{plain}
\markboth{JOCHEN~SIEHR AND DIRK~LEBIEDZ}
         {A CONTINUATION METHOD IN MODEL REDUCTION}

\section{Introduction} \label{s:intro}
Models including multiple time scales arise in many applications as e.g.\
combustion chemistry and biochemistry. The models under consideration are
described by a system of ordinary differential equations (ODE). Model reduction
methods for stiff multiple time scale ODE aim at an identification of the slow
directions in the phase space of the system. After a short relaxation phase, the
dynamics of the ODE model are mainly determined by the slow modes.

The slow dynamics are often described by so-called slow invariant manifolds
(SIM). In the phase space of the dynamical system under consideration,
trajectories bundle on those SIM being attractors of nearby trajectories. These
are hierarchically ordered with decreasing dimension towards stable equilibrium,
and the dimension corresponds to the set of slow (large) time scales. Fenichel
analyzes slow manifolds in context of singular perturbed systems of ODE in a
series of articles \cite{Fenichel1972,Fenichel1974,Fenichel1977,Fenichel1979}.
It is shown there that for a sufficiently large separation of time scales
(a sufficiently small singular perturbation parameter) there exists a mapping
from a compact subset of the subspace of slow variables into the subspace of
fast variables. The graph of this mapping is the slow manifold.

Many model reduction methods identify or approximate such SIM. The quasi steady
state approximation identifies the slow manifold of singular perturbed systems
of ODE by setting the singular
perturbation parameter to zero \cite{Bodenstein1907,Chapman1913,Segel1989}.
The dynamics are only described by the differential algebraic equation that
consists of the ODE for the slow variables and the implicit function relating
the fast variables to the slow ones and defining the slow manifold. The ILDM
method identifies a local SIM approximation based on an eigenvalue decomposition
of the Jacobian of the right hand side of the ODE \cite{Maas1992}. The CSP
method identifies a local basis where fast and slow dynamics in the phase space
are separated \cite{Lam1994}. Many more methods for model reduction based on an
analysis of a separation of time scales exist, see e.g.\ the review
\cite{Gorban2004}.

In this article, we discuss the development of an efficient application of a
model reduction method based on optimization of trajectories to models including
realistic detailed chemical combustion mechanisms. In Section~\ref{s:opt_more},
we review the optimization problem for model reduction. The numerical
optimization method to solve this optimization problem is explained in
Section~\ref{s:opt}. A continuation method including the optimization method as
corrector is introduced in Section~\ref{s:continuation}. Existence of a homotopy
path is regarded as well as a step size strategy for efficient progress. In
Section~\ref{s:results}, results of an application to example problems are
shown. The article is summarized in Section~\ref{s:conclusion} and an outlook
is given.

\section{Optimization based model reduction} \label{s:opt_more}
An approach for model reduction based on optimization of trajectories is raised
in \cite{Lebiedz2004c}. The idea is to identify certain trajectories along which
a criterion takes its minimum value that represents a mathematical
characterization of slowness and attraction of nearby trajectories. This
relaxation criterion is minimized subject to the constraints of the dynamics,
the fixed parameters for parametrization of the SIM, and eventual conservation
laws.

\subsection{Optimization problem}
In \cite{Lebiedz2011a}, it is shown for two specific models that the method
identifies asymptotically the correct SIM. Here we stick near to the
formulation chosen there. We denote the ODE that describe the dynamics of the
kinetic system with
\begin{equation} \label{eq:dyn}
 \D z(t) = S(z(t)),
\end{equation}
where $\D$ is the derivative of state vector $z(t)$ at time $t$ w.r.t.\ $t$. The
function $S:\D\rightarrow\RR^n$ is assumed sufficiently smooth in an open domain
$D\in\RR^n$. Conservation relations are valid for many models. Typically,
these are restrictions to an initial value $z(t_0)$ for an initial value problem
with the dynamics~(\ref{eq:dyn}) and represent the conservation of the mass of
chemical elements in the system.

A subset of state variables is usually chosen for parametrization of the SIM
approximation. In our formulation, these variables
$z_{j}(t_*)$, $j \in \III_{\text{pv}}$, which are
called \emph{reaction progress variables} or \emph{represented species}, are
fixed at $t_*$ to a certain value $z_j^{t_*},$ $j \in \III_{\text{pv}}$ with
$|\III_{\text{pv}}|<n$. The aim of a model reduction method that allows for
\emph{species reconstruction} is the computation of the corresponding
unrepresented variables $z_{j}(t_*)$, $j \notin \III_{\text{pv}}$, at $t_*$
leading to a local representation of the SIM.

A general formulation of the optimization problem (similar to
\cite{Lebiedz2011a}) to compute an approximation of a SIM  is given by
\begin{subequations}\label{eq:gen_op}\label{eq:iop}
 \begin{equation}\label{eq:gen_op:of}
  \min_{z}\ \Psi(z)
 \end{equation}
 subject to
  \begin{align}
   \D z(t) &= S(z(t)) \label{eq:gen_op:dyn} \\
         0 &= \bar{C}(z(t)) \label{eq:gen_op:cons}\\
         0 &= z_j(t_*) - z_j^{t_*},\quad j \in \III_{\text{pv}} \label{eq:gen_op:pv}\\
         0 &\leqslant z(t) \label{eq:gen_op:pos}
  \end{align}
 and
 \begin{align}
  t   &\in [t_0,t_\f] \\
  t_* &\in [t_0,t_\f] \quad \text{\upshape (fixed)}.
  \end{align}
\end{subequations}
This means, a trajectory (piece) $z$ has to be identified on an interval
$[t_0,t_\f]$ such that the trajectory obeys the ODE
(\ref{eq:gen_op:dyn}) and necessary conservation relations
(\ref{eq:gen_op:cons}). The reaction progress variables are fixed to the desired
values $z_{j}^{t_*}$, $j \in \III_{\text{pv}}$ at a point in time $t_*$ in the
interval $[t_0,t_\f]$ in Eq.~(\ref{eq:gen_op:pv}). In some applications,
positivity of the state variables has to be demanded explicitly to guarantee
physical feasibility in the iterative solution algorithm for problem
(\ref{eq:gen_op}) as it is done in (\ref{eq:gen_op:pos}).

\subsubsection{Objective function}
Various suggestions for a suitable objective functional $\Psi$ have been made,
especially in \cite{Lebiedz2010}. In \cite{Lebiedz2011a}, the objective function
\begin{equation}\label{eq:of}
 \Psi(z) \coloneqq \int_{t_0}^{t_{\textrm f}}  \Phi\left(z(t)\right) \; \d t
\end{equation}
with the integrand
\begin{equation*}
 \Phi(z(t)) = \| J_{S}(z(t))\;S(z(t)) \|_2^2
\end{equation*}
in the Lagrange term is used, where $J_{S}$ denotes the Jacobian of $S$.

This might be motivated by the estimate
\begin{equation*}
 \| J_{S}\;S \|_2 \leqslant \| J_{S}\|_2\; \| S \|_2,
\end{equation*}
where a small spectral norm $\| J_{S}\|_2$ relates to the attraction property of
the SIM and a small $\| S \|_2$ relates to slowness.

\subsubsection{Local formulation}
The approximation of points on the SIM computed as solution of optimization
problem (\ref{eq:gen_op}) is often used e.g.\ in computational fluid dynamics
(CFD) and other applications. In this context, a large number of approximation
points of the SIM for different values of the reaction progress variables is
needed. This means, problem (\ref{eq:gen_op}) has to be solved repeatedly for a
range of values of $z_j^{t_*}$, $j \in \III_{\text{pv}}$. This is generally very
time consuming.

In this case, a local formulation (in time) can be used. It is given by
(cf.\ \cite{Girimaji1998})
\begin{subequations}\label{eq:local_op}
 \begin{equation}
  \min_{z(t_{*}),T(t_{*})} \; \| J_{S}(z(t_*))\;S(z(t_*)) \|_2^2
 \end{equation}
 \upshape{subject to}
 \begin{align}
  0 &= C(z(t_{*})) \\
  0 &= z_j(t_{*}) - z_j^{t_{*}},\quad j \in \III_{\text{pv}} \\
  0 &\leqslant z(t_{*}) \label{eq:local_op:pos}.
 \end{align}
\end{subequations}

In the case of the objective function $\Psi$ in (\ref{eq:of}), optimization problem
(\ref{eq:gen_op}) is semi-infinite. A discretization method such as a collocation
of orthogonal polynomials or a shooting approach has to be applied to project it
into a finite nonlinear programming (NLP) problem. If
(\ref{eq:local_op}) is used, the optimization problem is a finite dimensional
NLP problem, and the solution of the system dynamics (\ref{eq:gen_op:dyn}) is
dispensable. This problem can be solved directly with standard NLP software
as sequential quadratic programming \cite{Powell1978} or interior point
\cite{Forsgren2002} methods.

\subsection{Parametric optimization} \label{s:param_opt}
In the optimization problem which is solved for model reduction purposes, there
is a parameter dependence in terms of the fixation of the reaction progress
variables. Therefore, we consider parametric finite dimensional NLP problems in
the following. The standard problem is given as
\begin{subequations} \label{eq:fd_op_param}
 \begin{equation}
  \min_{x\in\RR^n}  f(x,r)
 \end{equation}
 subject to
 \begin{align}
    g(x,r) &= 0        \\
    h(x,r) &\geqslant 0, \label{eq:fd_op_param:ineq}
 \end{align}
\end{subequations}
where the objective function $f: D\times\tilde{D} \rightarrow \RR$,
the equality constraint function $g:D\times\tilde{D} \rightarrow \RR^{n_2}$,
and the inequality constraint function
$h:D\times\tilde{D} \rightarrow \RR^{n_3}$ depend on the parameter vector
$r \in \tilde{D}$ with both $D \subset \RR^n$, $\tilde{D} \subset \RR^{n_r}$ open.
The Lagrangian function can be written as
\begin{equation*}
 \LLL(x,\lambda,\mu,r) \coloneqq f(x,r) - \lambda^\T g(x,r) - \mu^\T h(x,r).
\end{equation*}

First order sensitivity results are given by the following theorem of Fiacco
\cite{Fiacco1976}. These results are used in context of real-time
optimization, see e.g.\ \cite{Bueskens2001}, and nonlinear model predictive
control, e.g.\ in \cite{Ferreau2008,Zavala2009}, for embedding of neighboring
solutions of a function of a parameter $r$.

\begin{theorem}[Second-order sufficient conditions \cite{Nocedal2006}]\label{thm:ssc}
 Let $x^*$ be a feasible point of (\ref{eq:fd_op_param}) for which the KKT
 conditions are satisfied with Lagrange multipliers $(\lambda^*,\mu^*)$. Suppose
 further that
 \begin{equation*}
  v^\T \nabla_{\!xx\,}^2 \LLL(x^*, \lambda^*, \mu^*,r) v > 0 \quad \forall v \in \CCC(x^*, \lambda^*, \mu^*),\ v\neq0,
 \end{equation*}
 where $\CCC$ is the \emph{critical cone}.
 Then $x^*$ is a strict local minimizer for (\ref{eq:fd_op_param}).
\end{theorem}
\begin{proof}See e.g.\ \cite{Fletcher1981,Nocedal2006}.
\end{proof}

\begin{theorem}[Parameter sensitivity \cite{Fiacco1976}]\label{thm:fiacco-sens}
 Let the functions $f$, $g$, and $h$ in problem (\ref{eq:fd_op_param}) be twice
 continuously differentiable in a neighborhood of $(x^*,0)$. Let the second
 order sufficient conditions hold (see Theorem~\ref{thm:ssc}) for a local
 minimum of (\ref{eq:fd_op_param})
 at $x^*$ with $r=0$ and Lagrange multipliers $\lambda^*$, $\mu^*$. Furthermore,
 let linear independence constraint qualification (LICQ) be valid in $(x^*,0)$
 and strict complementary slackness, i.e.\ $\mu_i^*>0$ if
 $h_i(x^*,0)=0$, $i=1,\dots,n_3$. Then the following holds:
 \begin{enumerate}[\rm (i)]
  \item The point $x^*$ is a isolated local minimizer of (\ref{eq:fd_op_param})
        with $r=0$, and the associated Lagrange multipliers $\lambda^*$, $\mu^*$
        are unique.
  \item For $r$ in a neighborhood of $0$, there exists a unique once
        continuously differentiable function $(x(r),\lambda(r),\mu(r))$
        satisfying the second order sufficient conditions for a local minimum of
        (\ref{eq:fd_op_param}) such that
        \begin{equation*}
         (x(0),\lambda(0),\mu(0))=(x^*,\lambda^*,\mu^*),
        \end{equation*}
        and $x(r)$ is a isolated local minimizer of (\ref{eq:fd_op_param}) with
        Lagrange multipliers $\lambda(r)$ and $\mu(r)$.
  \item Strict complementarity with respect to $\mu(r)$ and LICQ hold at $x(r)$
        for $r$ near $0$.
 \end{enumerate}
\end{theorem}
\begin{proof}See \cite{Fiacco1976}.
\end{proof}

Numerically, the inequality constraints (\ref{eq:fd_op_param:ineq}) can be
treated with either an active set (AS) strategy or an interior point (IP)
method. In an AS strategy, the AS is updated in each iteration. Active
constraints are considered as equality constraints, inactive constraints are
omitted. By contrast, the objective function is modified with a barrier term
eliminating the inequality constraints in an IP framework. Therefore, we only
regard equality constraints in the following.

The necessary optimality (KKT) conditions (stationarity of the Lagrangian
function, feasibility, and complementarity) for problem (\ref{eq:gen_pop}) can
shortly be written as
\begin{equation}\label{eq:kkt_simple}
 K(x,\lambda,\mu,r) = 0.
\end{equation}
with e.g.\ for an active set method
\begin{equation*}
K(x,\lambda,\mu,r) \coloneqq
\left\{
\begin{aligned}
 &\nabla f(x,r) - \nabla_{\!x\,} g(x,r)\lambda - \nabla_{\!x\,} h(x,r)\mu\\
 &g(x,r)\\
 &h_i(x,r),\quad i\in\AAA(x),
\end{aligned}
\right.
\end{equation*}

We are interested in the derivative of the solution $(x^*(r),\lambda^*(r),\mu^*(r))$
of (\ref{eq:kkt_simple}) with respect to the
parameters. The implicit function theorem yields
\begin{equation*}
 \D_{(x,\lambda,\mu)} K(x^*,\lambda^*,\mu^*,r) \; \D_r (x,\lambda,\mu)
 = - \D_r K(x^*,\lambda^*,\mu^*,r).
\end{equation*}
The symbol $\D_x f$ denotes the partial derivative of $f$ w.r.t.\ $x$.
The same equation in matrix notation is
\begin{equation}\label{eq:KKT_sens_gen}
 \begin{bmatrix}
  \D_x K & \D_{\lambda}K & \D_{\mu}K
 \end{bmatrix}
 \begin{bmatrix}
  \D_r x \\ \D_r \lambda \\ \D_r \mu
 \end{bmatrix}
 = - \D_r K.
\end{equation}
The KKT matrix 
\begin{equation}
 \begin{bmatrix}
  \D_x K & \D_{\lambda}K & \D_{\mu}K
 \end{bmatrix}
\end{equation}
is nonsingular if LICQ and second order sufficient optimality conditions
\cite{Fletcher1981,Nocedal2006} are fulfilled.

\section{Numerical solution of the optimization problem (\ref{eq:local_op})} \label{s:opt}\label{s:GGN}
The optimization problem (\ref{eq:local_op}) is a constrained nonlinear least
squares (CNLLS) problem of the form
\begin{subequations}\label{eq:CNLLS}
 \begin{equation}
  \min_{x\in\RR^n}\;\tfrac{1}{2}\| F_1(x) \|_2^2
 \end{equation}
 \upshape{subject to}
 \begin{align}
   F_2(x) &= 0 
 \end{align}
\end{subequations}
where the functions $F_i: D\rightarrow\RR^{n_i},\ i=1,2$, are
supposed to be at least twice continuously differentiable in their open domain
$D\subset\RR^n$. The problem is solved
iteratively: $x_{k+1} = x_k + t_k d_k$, where $d_k\in\RR^n$ is
the increment and $t_k\in(0,1]$ the step length.

In order to solve this CNLLS problem, we use a generalized Gauss--Newton
(GGN) method as described in \cite{Bock1987,Stoer1971}, where the step direction
(increment) is computed as solution of the linearized optimization problem
\begin{subequations}\label{eq:CLLS}
 \begin{equation}\label{eq:CLLS:of}
  \min_{d\in\RR^n}\;\tfrac{1}{2}\| F_1(x_k) + J_1(x_k) d \|_2^2
 \end{equation}
 \upshape{subject to}
 \begin{align}
   F_2(x_k) + J_2(x_k)d &=         0  \label{eq:CLLS:ce}
 \end{align}
\end{subequations}
where $J_i$ is the Jacobian matrix of $F_i$ for $i=1,2$.

An AS strategy is used to treat inequality constraints. For globalization,
we employ a filter method \cite{Fletcher2002}. The filter is implemented as
in \cite{Waechter2005,Waechter2006}. The most important feature of a filter
is the fact, that a step is accepted if it improves the value of the objective
function or the constraint violation. The combination of a GGN
method with a filter method for globalization of convergence is new to our
knowledge.

To prevent the Maratos effect \cite{Nocedal2006}, we use a second order
correction (SOC) as in \cite{Waechter2005,Waechter2006}. The SOC increment
is computed as solution of the least squares problem, cf.\ \cite{Nocedal2006}:
\begin{subequations}
 \begin{equation*}
  \min_{\check{d}}\;\| \check{d} \|_2^2
 \end{equation*}
 \upshape{subject to}
 \begin{equation*}
  F_2(x+d) + J_2(x)\check{d}=0.
 \end{equation*}
\end{subequations}

If a trial point is not accepted by the filter after several reductions of
the step size, a feasibility restoration phase is necessary. This can be done
efficiently in context of a Gauss-Newton method. The aim of the feasibility
restoration phase is the computation of a feasible point that is ``near'' to the
last accepted iterate and acceptable to the updated filter.

The goal to find a feasible point that is ``near'' to the last iterate $x=x_k$
can be written in context of GGN methods as the following CNLLS
problem
\begin{subequations}\label{eq:resto-CNLLS}
 \begin{equation}
  \min_{\bar{x}}\;\tfrac{1}{2}\| \bar{x} - x \|_2^2
 \end{equation}
 \upshape{subject to}
 \begin{align}
   F_2(\bar{x}) &= 0 
 \end{align}
\end{subequations}
Problem (\ref{eq:resto-CNLLS}) can be solved iteratively with an increment
$\bar{d}_{kl}$ with the new iteration index $l$ as solution of a CLLS
optimization problem
\begin{subequations}\label{eq:resto-CLLS}
 \begin{equation}
  \min_{\bar{d}}\;\tfrac{1}{2}\| \bar{x}_k -x + \bar{d} \|_2^2
 \end{equation}
 \upshape{subject to}
 \begin{align}
   F_2(\bar{x}_k) + J_2(\bar{x}_k)\bar{d} &=         0 
 \end{align}
\end{subequations}
and initial value $\bar{x}^0 \coloneqq x$. The only difference between
(\ref{eq:resto-CLLS}) and problem (\ref{eq:CLLS}) is the objective function,
matrix factorizations, e.g.\ of $J_2$, can be reused.

\section{Continuation strategy} \label{s:continuation}
It is useful to follow the homotopy path of the zero of the KKT conditions
in dependence of the reaction progress variables to solve families of
optimization problems alike (\ref{eq:gen_op}) for different parameters.

In the following, we consider the finite dimensional parametric optimization
problem
\begin{subequations}\label{eq:gen_pop}
 \begin{equation}\label{eq:gen_pop:of}
  \min_{x\in\RR^n}\ f(x)
 \end{equation}
 subject to
 \begin{align}
  0 &= g(x) \label{eq:gen_pop:eq}\\
  0 &= x_{j(i)} - r^i,\quad i=1,\dots,n_r \label{eq:gen_pop:pv} 
 \end{align}
\end{subequations}
where the functions $f:D\rightarrow \RR$ and $g:D\rightarrow \RR^{n_2}$ are
$\CCC^2(D)$ in the open domain $D\subset\RR^n$.
The fixed values of the reaction progress
variables in (\ref{eq:gen_pop:pv}) are denoted by the parameter vector
$r\in\RR^{n_r}$, $r^i=z_{j(i)}^{t_{*}}$, $i=1,\dots,n_r$, $n_r<n-n_2$ with the
notation of
Equation~(\ref{eq:gen_op:pv}), where
$j:\{1,\dots,n_r\}\rightarrow \III_{\text{pv}}$, $i\mapsto j(i)$ is a bijective
map to the index set $\III_{\text{pv}}\subset\{1,\dots,n\}$ of the reaction
progress variables.


\subsection{Parameter sensitivities in the GGN method}
If Newton's method together with an active set method is used to find a
candidate solution of (\ref{eq:gen_pop}), i.e.\ the
computation of a root of $K$, the KKT matrix
$\D_{(x,\lambda,\mu)} K(x_k,\lambda_k,\mu_k)$ has to be computed in every
iteration $k$. This is different if a generalized Gauss--Newton method is
employed.

We return to the notation used in Section~\ref{s:opt}: The objective function
is written as $f(x)=\|F_1(x)\|_2^2$ with a sufficiently smooth function
$F_1:D\subset \RR^n\rightarrow\RR^{n_1}$. The constraints are given as
$F_2:D\subset \RR^n\rightarrow\RR^{\bar{n}_2}$,
$\bar{n}_2=n_2 + n_r + |\AAA(x)|$ with
\begin{equation*}
F_2(x) =
\left\{
\begin{aligned}
 &g(x)\\
 &x_{j(i)} - r^i,\quad i=1,\dots,n_r\\
 &x_i,\quad i\in\AAA(x)
\end{aligned}
\right.
\end{equation*}
and the active set $\AAA(x)$ in $x$. The Jacobian matrices of $F_1$ and $F_2$
are denoted as $J_1$ and $J_2$, respectively.

In every iteration, the equation system
\begin{equation} \label{eq:kkt-clls}
 \begin{bmatrix}
  J_1^\T J_1 & J_2^\T \\ J_2 & 0
 \end{bmatrix}
 \begin{bmatrix}
  d\\-\lambda
 \end{bmatrix}
 = -
 \begin{bmatrix}
  J_1^\T F_1 \\ F_2
 \end{bmatrix}
\end{equation}
is solved in the GGN method (where the argument $x_k$ is omitted and the vector
$\lambda$ is used for all equality and active inequality constraints);
that is the KKT system of the CLLS problem (\ref{eq:CLLS}), see
Section~\ref{s:opt}. By contrast,
\begin{equation} \label{eq:kkt-cnlls}
 \begin{bmatrix}
  \LLL_{xx} & -J_2^\T \\ J_2 & 0
 \end{bmatrix}
 \begin{bmatrix}
  \Delta x\\ \Delta\lambda
 \end{bmatrix}
 = -
 \begin{bmatrix}
  \nabla_x \LLL \\ F_2
 \end{bmatrix}
\end{equation}
is solved if Newton's method is applied to find a KKT point of the original
CNLLS problem (\ref{eq:CNLLS}). The derivative of the Lagrangian function of the
CNLLS problem is
$\nabla_{\!x\,} \LLL ^\T = F_1(x)^\T J_1(x) - \lambda^\T J_2(x)$. The difference
in the KKT matrices in Equations~(\ref{eq:kkt-cnlls}) and (\ref{eq:kkt-clls}) is
the difference in the Hessians of the different Lagrangian functions: In the GGN
method, $J_1^\T J_1$ is used; whereas Newton's method applied to the KKT
conditions of the original CNLLS problem (\ref{eq:CNLLS}) exploits
\begin{equation}\label{eq:Hess_Lagrange_CNLLS}
\begin{aligned}
 \LLL_{xx} &= \nabla_{\!xx\,}\tfrac{1}{2}\|F_1(x)\|_2^2 - \sum_{i=1}^{\bar{n}_2}\lambda_i \nabla_{\!xx\,} F_2^i(x) \\
           &= J_1^\T J_1 + \left(\D_x J_1^\T\right)F_1 - \sum_{i=1}^{\bar{n}_2}\lambda_i \nabla_{\!xx\,} F_2^i(x),
\end{aligned}
\end{equation}
where $F_2^i$ is the $i$-th component of $F_2$.

In our application in chemical kinetics, the constraints $F_2$ are given
via conservation relations. In this case, the second derivative
$\nabla_{\!xx\,} F_2^i$ of $F_2^i$
typically is zero. This might not be true, if an entry in $F_2$ is a nonlinear
equation in $x$ representing an energy conservation law, see also
\cite{Lebiedz2013}. 
This means, the effort for computing  $\LLL_{xx}$ instead of $J_1^\T J_1$ is
mainly the effort to compute the expression $\left(\D_x J_1^\T\right)F_1$.
This can be evaluated using automatic differentiation \cite{Griewank2000} as a
directional derivative of $J_1$ with respect to direction $F_1$.

\subsection{Euler predictor}\label{ss:derivative-pred}\label{s:step_size_strategy}
If the approximation of points on the SIM is used in simulations in
computational fluid dynamics, approximations of the tangent vectors of the
SIM in these points are needed, too, see e.g.\ \cite{Ren2007b}. These are
given via the parameter sensitivities
$D_{r^i} x^* = \tfrac{\d z^*(t_*)}{\d z_{j(i)}^{t_*}}$ at
the solution $x^*$ and $z^*$ of the optimization problems (\ref{eq:gen_pop}) and
(\ref{eq:gen_op}), respectively.
As these derivatives are available, we use the Euler prediction in a predictor
corrector scheme for initialization of the optimization algorithm (corrector)
to solve neighboring problems. In our case, this is the computation of a
solution with the optimization algorithm for various parameter values of the
reaction progress variables $r$.

The prediction can be used in a homotopy method with a step size strategy.
An effective step size strategy is published in \cite{Heijer1981} and
extensively discussed and modified in
\cite{Allgower1980,Allgower1990,Allgower1997}. We use the method as discussed
in \cite{Allgower1990}.
The aim of the step length strategy of den Heijer and Rheinboldt
\cite{Heijer1981} is to achieve a desired number of iterations for the corrector
step. The strategy allows for the computation of
a step length based on the contraction rate of the latest corrector iterations
and an error model for the corrector such that the desired number of iterations
is achieved.

We define a curve $c$ as a mapping from the parameter space to the space of the
primal and dual variables of (\ref{eq:gen_pop})
\begin{align*}
 c:\RR^{n_r} &\rightarrow \RR^{n+n_2} \\
           r &\mapsto c(r)\coloneqq 
           \begin{pmatrix}
           x^*(r)\\ \lambda^*(r)
           \end{pmatrix}.
\end{align*}
For each value of the parameter vector $r$, $c(r)$ is the solution
$(x^*(r),\lambda^*(r))$ of (\ref{eq:gen_pop}).

We denote the Euler predictor step
\begin{equation} \label{eq:Euler-prediction}
  c_{i+1}^0(h_i) = c_{i} + h_i(r_{i+1} - r_i)^\T \tfrac{\d}{\d r}c_{i}(r_i),
  \quad i=0,\dots
\end{equation}
The prediction $c_i^0$ is used as initialization of the corrector to compute the
solution of (\ref{eq:gen_pop}). The corrector iterations are denoted
\begin{equation*}
 c_i^{j+1}(h_i) = \mathrm{C}(c_i^j(h_i)),\quad j=0,\dots
\end{equation*}
with the corrector $\mathrm{C}$. It is assumed that the corrector
iterations converge to the solution
\begin{equation*}
 c_i^\infty (h_i) \coloneqq \lim_{j\rightarrow\infty} c_i^j(h_i) \in K^{-1}(0).
\end{equation*}

The sophisticated aspect in the work of den Heijer and Rheinboldt
\cite{Heijer1981} is the error model $\phi$. This
error model estimates the error of the iterates
\begin{equation*}
 \epsilon_j(h_i) = \| c_i^\infty (h_i) - c_i^{j}(h_i) \|_2
\end{equation*}
independently of $h$ via an expression of the form
\begin{equation*}
 \epsilon_{j+1}(h_i) \leqslant \phi(\epsilon_j(h_i)).
\end{equation*}
The formula for the error model depends on the contraction rate of the
corrector, see \cite[p.~53ff.]{Allgower1990}, where we use the error model for the
quadratically convergent Newton's method and for the linearly convergent GGN
method.

\paragraph{Linear step}
Optimization problem (\ref{eq:gen_pop}) has to be solved
many times for different values of the parameter $r$. Especially if the
approximation of points on a SIM is needed \emph{in situ}, e.g.\ in a CFD
simulation, it is necessary to compute the solution of neighboring optimization
problems fast.

If the values $r^{\textrm{new}}$ for the reaction progress variables, for which
a SIM approximation is needed, are near to the values $r^*$
\begin{equation} \label{eq:linear_step_tol}
 \| r^{\textrm{new}} - r^* \|_2 < \epsilon_{\mathrm{tol}}
\end{equation}
for which the optimization problem is already solved, it can save computing time
to use $z^{\lin}$ as approximation of the SIM. This is defined (in analogy to
(\ref{eq:Euler-prediction})) as
\begin{equation} \label{eq:linear_step}
 z^{\lin} \coloneqq z^*(t_*) + \left( r^{\textrm{new}} - r^*\right)^\T \frac{\d z^*(t_*)}{\d r}(r^*)
\end{equation}
where the notation is in accordance with the notation in (\ref{eq:iop}),
$r^i=z_{j(i)}^{t_{*}}$, $i=1,\dots,n_r$, $n_r=|\III_{\text{pv}}|$,
$j$ is the bijection, and $r^i$ is the $i$-th component of $r$.

\section{Numerical results} \label{s:results}
For numerical validation of our method, we use a test model for model reduction
purposes. The reaction mechanism is given in Table~\ref{t:ICE}. We use
thermodynamical data in form of coefficients of NASA polynomials we received
from J.~M.~Powers and A.~N.~Al-Khateeb, which they use in
\cite{Al-Khateeb2010,Al-Khateeb2009}.

The mechanism is published originally in \cite{Li2004}. The
simplified version shown in Table~\ref{t:ICE} is used by Ren et al.\ in
\cite{Ren2006a}. The mechanism consists of five reactive species and
inert nitrogen, where in comparison to a full hydrogen combustion mechanism the
species $\species{O}_2$, $\species{HO}_2$, and $\species{H}_2\species{O}_2$ are
removed. The species are involved in six Arrhenius type reactions, where three
combination/decomposition reactions require a third body for an effective
collision.
\begin{table} \centering
 \caption[Simplified mechanism]
  {\label{t:ICE}Simplified mechanism as used in \cite{Ren2006a}. Collision
  efficiencies $\species{M}$:
  $\alpha_{\species{H}} = 1.0 , \alpha_{\species{H_2}} = 2.5, \alpha_{\species{OH}} = 1.0$,
  $\alpha_{\species{O}} = 1.0 , \alpha_{\species{H_2O}} = 12.0, \alpha_{\species{N_2}} = 1.0$;
  \cite{Ren2006a}.}
 \begin{tabular}{@{}lclrrr@{}}
 \toprule
 \multicolumn{3}{@{}l}{Reaction} & $A$ / ($\textrm{cm}, \textrm{mol}, \textrm{s}$) & $b$ & $E_\textrm{a}$ / $\si{\kilo\joule\per\mole}$ \\
 \cmidrule(r){1-3} \cmidrule(lr){4-4} \cmidrule(lr){5-5} \cmidrule(l){6-6}
 $\species{O} + \species{H_2} $ & $\rightleftharpoons$ & $\species{H} + \species{OH}$ &  $5.08\times 10^{04}$ & $2.7$ & $26.317$  \\
 $\species{H_2} + \species{OH}$ & $\rightleftharpoons$ & $\species{H_2O} + \species{H}$ & $2.16\times 10^{08}$ & $1.5$ & $14.351$ \\
 $\species{O}$ + $\species{H_2O}$ & $\rightleftharpoons$ & $\species{2}\,\species{OH}$ & $2.97\times 10^{06}$ & $2.0$ & $56.066$ \\
 $\species{H_2} + \species{M}$ & $\rightleftharpoons$ & $\species{2}\,\species{H} + \species{M}$ & $4.58\times 10^{19}$ & $-1.4$ & $436.726$ \\
 $\species{O} + \species{H} + \species{M}$ & $\rightleftharpoons$ & $\species{OH} + \species{M}$ & $4.71\times 10^{18}$ & $-1.0$ & $0.000$ \\
 $\species{H} + \species{OH} + \species{M}$ & $\rightleftharpoons$ & $\species{H_2O} + \species{M}$ & $3.80\times 10^{22}$ & $-2.0$ & $0.000$ \\
 \bottomrule
\end{tabular}
\end{table}
The reaction system is considered under isothermal and isobaric conditions at a
temperature of $T=\SI{3000}{\kelvin}$ and a pressure of
$p=\SI{101325}{\pascal}$. Hence, the state of the system is sufficiently
described by the specific moles of the chemical species
$z_i$, $i=1,\dots,n_{\textrm{spec}}$ in $\si{\mole\per\kilogram}$.

Conservation relations for the elemental mass in this model are given in terms
of amount of substance as \cite{Ren2006a}
\begin{equation}
\begin{aligned} \label{eq:ice-conserve}
 n_\species{H}+2\,n_{\species{H}_2}+n_\species{OH}+2\,n_{\species{H}_2\species{O}} &= \SI{1.25e-3}{\mole} \\
 n_\species{OH}+n_\species{O}+n_{\species{H}_2\species{O}} &= \SI{4.15e-4}{\mole} \\
 2\,n_{\species{N}_2} &= \SI{6.64e-3}{\mole}.
\end{aligned}
\end{equation}
The total mass in the system can be computed with the values in
Equation~(\ref{eq:ice-conserve}) and has a value of $m=\SI{1.01e-4}{\kilogram}$.

\subsection{One-dimensional SIM approximation} \label{ss:res_1d}
Results of an approximation of a one-dimensional SIM are shown in
Figure~\ref{f:result1} and \ref{f:result2}.
We use $z_{\species{H_2O}}$ as reaction progress variable and vary its value.
The values of all remaining species at the solution of the different
optimization problems are plotted versus $z_{\species{H_2O}}$ in form of
trajectories through the solution points $z^*(t_*)$ which are shown as x marks.
The computation is done with the GGN method as described in Section~\ref{s:opt}.

We use values near the equilibrium state as initial values in the optimization
algorithm as we assume this point near a slow manifold. We set the values of
the reaction progress variable near its initial value;
see Table~\ref{t:init}. This allows for a fast computation of a solution of the first
optimization problem with $z_{\species{H}_2\species{O}}^{t_*}=\SI{3}{\mole\per\kilogram}$.

\begin{table}[htbp] \centering
 \caption[Simplified hydrogen combustion: initial values and solution]
  {\label{t:init}Initial value (unscaled) for the algorithm and a solution of the
   optimization problem (\ref{eq:iop}) as solved for the results depicted
   in Figure~\ref{f:result1} to reduce the simplified hydrogen
   combustion model with $z_{\species{H}_2\species{O}}^{t_*}=\SI{3}{\mole\per\kilogram}$.}
 \begin{tabular}{@{}lll@{}}
 \toprule
   Variable &  Initial value $z^0(t_*)$ & Numerical solution $z^*(t_*)$ \\
  \cmidrule(r){1-1}  \cmidrule(lr){2-2} \cmidrule(l){3-3}
   $z_{\species{O}}$                  & $\SI{0.34546441}{}$ & $\SI{0.34563763}{}$\\
   $z_{\species{H}_2}$                & $\SI{2.0279732}{}$  & $\SI{2.0281615}{}$\\
   $z_{\species{H}}$                  & $\SI{1.5195639}{}$  & $\SI{1.5193606}{}$\\
   $z_{\species{OH}}$                 & $\SI{0.76454959}{}$ & $\SI{0.76437637}{}$\\
   $z_{\species{H}_2\species{O}}$     & $\SI{3.0000000}{}$  & $\SI{3.0000000}{}$\\
   $z_{\species{N}_2}$                & $\SI{32.905130}{}$  & $\SI{32.905130}{}$\\
 \bottomrule
\end{tabular}
\end{table}

All computations are done on an Intel$^\circledR$
Core$^{\textrm{TM}}$ i5-2410M CPU with 2.30\,GHz, operating system
openSUSE 12.2 (x86\_64) including the Linux kernel 3.4.11 and GCC 4.7.1.
We do not use a step size strategy in the predictor corrector scheme. To compute
the 17 points shown in Figure~\ref{f:result1}, 74 iterations are necessary,
which take in sum $\SI{0.02}{\second}$.
\begin{figure}[htb]
 \centering
 \includegraphics[width=0.75\textwidth]{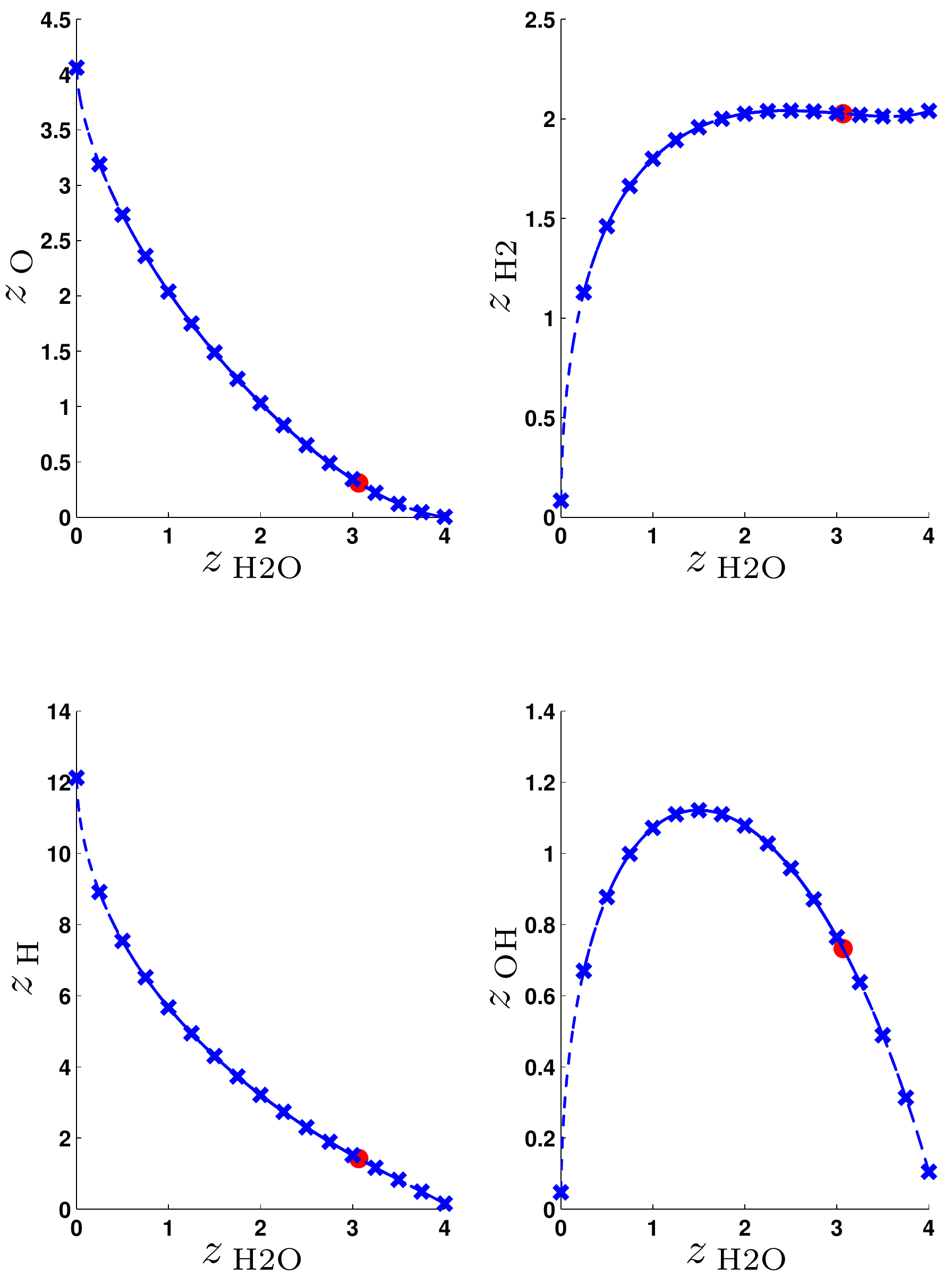}
 \caption{\label{f:result1}Numerical solutions of the
 optimization problem (\ref{eq:local_op}) to approximate a one-dimensional
 manifold in the state space of the six species model for hydrogen combustion.
 The full dot represents the equilibrium. The x marks depict solutions
 $z^*(t_*)$ of (\ref{eq:local_op}) for different values of the reaction progress
 variables $z_{\species{H_2O}}$ at $t_*$. Trajectories through these points
 (curves in the figures) are also shown.}
\end{figure}

Figure~\ref{f:result2} is presented to illustrate the linear step approximation.
For the computation of the three solutions (as $\epsilon_{\mathrm{tol}}=1.1$ is
chosen) of the optimization problem with different values for $z_{\species{H_2O}}$, 15
iterations are needed (in sum), which is slightly more effort per optimization
problem than in the case, where the results are illustrated in
Figure~\ref{f:result1}, where the distance between the different values
for $z_{\species{H_2O}}^{t_*}$ is only $0.25$. It can be seen that
the result of the linear approximation can lead to large deviations from a
smooth, invariant manifold.
\begin{figure}[htb]
 \centering
 \includegraphics[width=0.75\textwidth]{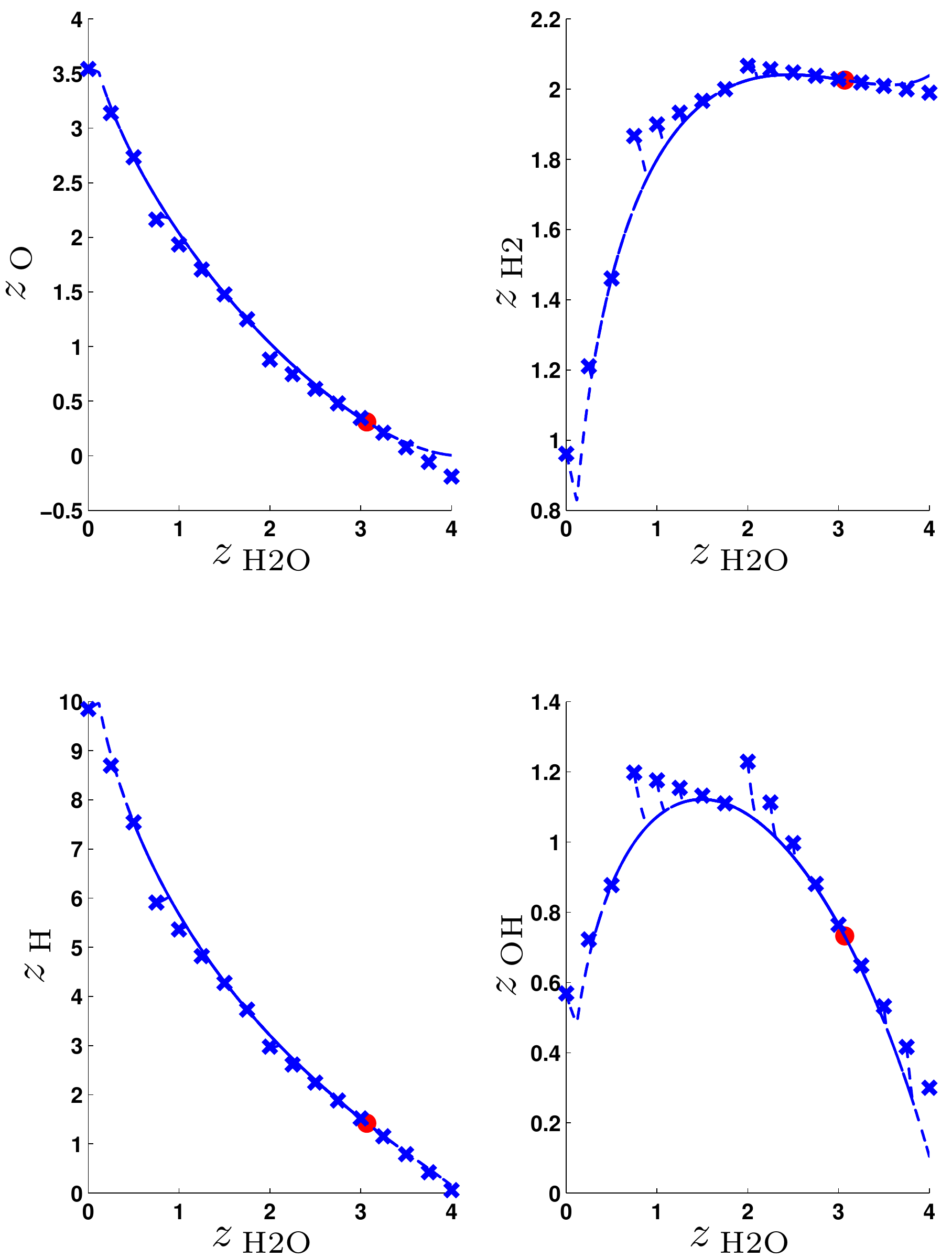}
 \caption{\label{f:result2}Illustration of solutions of the same problem as
 solved for the results shown in Figure~\ref{f:result1}, but with a linear
 step tolerance of $\epsilon_{\mathrm{tol}}=1.1$, see
 (\ref{eq:linear_step_tol}).}
\end{figure}

\paragraph{Warm start with step size strategy}
We want to demonstrate that a step size strategy for the warm start of the
algorithm to solve
neighboring optimization problems can have a benefit. So we consider the
optimization problem (\ref{eq:iop}) with $t_*=t_\f$, $t_\f-t_0=\SI{e-7}{\second}$
to be solved for the nominal parameter
$z_{\species{H}_2\species{O}}^{t_*}=\SI{3}{\mole\per\kilogram}$
with the shooting approach and \texttt{Ipopt} \cite{Waechter2006}. As neighboring
problem we consider the parameter
$z_{\species{H}_2\species{O}}^{t_*}=\SI{0.5}{\mole\per\kilogram}$. Such a large
change in the parameter can occur e.g.\ if the presented method for model
reduction is used \emph{in situ} in a CFD simulation and grid refinements
are performed in different regions of the spatial domain.

We solve the optimization problem first with $z_{\species{H}_2\species{O}}^{t_*}=\SI{3}{\mole\per\kilogram}$
and second with $z_{\species{H}_2\species{O}}^{t_*}=\SI{0.5}{\mole\per\kilogram}$
with a full step method in the predictor corrector scheme and apply the Euler
prediction.

The computations are done on an Intel$^\circledR$
Core$^{\textrm{TM}}$ i5-2410M CPU with 2.30\,GHz, operating system
openSUSE 12.2 (x86\_64) including the Linux kernel 3.4.11 and GCC 4.7.1.
Six iterations in \texttt{Ipopt} \cite{Waechter2006} have
to be preformed to solve the nominal optimization problem and nine iterations to
solve the second optimization problem. In sum, 15 iterations are necessary which
take $\SI{0.82}{\second}$ in sum.

If we initialize the step size strategy, see Section~\ref{s:step_size_strategy},
with initial step size $h_{\textrm{init}}=\tfrac{1}{2.5}$ and the desired number
of iterations $\tilde{k}=10$, we need one intermediate step in the predictor
corrector scheme described in
Section~\ref{ss:derivative-pred}. This is to solve the optimization problem with
the parameter $z_{\species{H}_2\species{O}}^{t_*}=\SI{2}{\mole\per\kilogram}$.
In this case we need in sum $6+5+8=19$ iterations that take only $\SI{0.70}{\second}$.
The difference in the computation time arises as the point for initialization of
the algorithm to solve the second optimization problem in the full step method
is not near the solution such that the KKT matrix is ill-conditioned. The inertia
correction in \texttt{Ipopt}, see \cite{Waechter2006},
is activated, and 19 line search iterations are performed in sum. In case of an
activated step size strategy, the initial value for the algorithm to solve the
optimization problem in case of a warm start is near the solution of the
optimization problem such that no inertia correction and also 19 line search
iterations (one per Newton iteration) are necessary in the presented example.

For an efficient tracking of the SIM for different values of the reaction
progress variables, it is important to stay close to the solution in neighboring
optimization problems. Therefore, a step size strategy is beneficial.

\subsection{Two-dimensional SIM approximation} \label{ss:res_2d}
We choose $z_1=z_{\species{H}_2}$ and $z_2=z_{\species{O}}$ as reaction progress
variables. Numerical solutions of (\ref{eq:iop}) for
the reduction of the simplified model for hydrogen combustion are shown in
Figure~\ref{f:ice2d}.
\begin{figure}[htbp]
 \centering
 \includegraphics[width=0.75\textwidth]{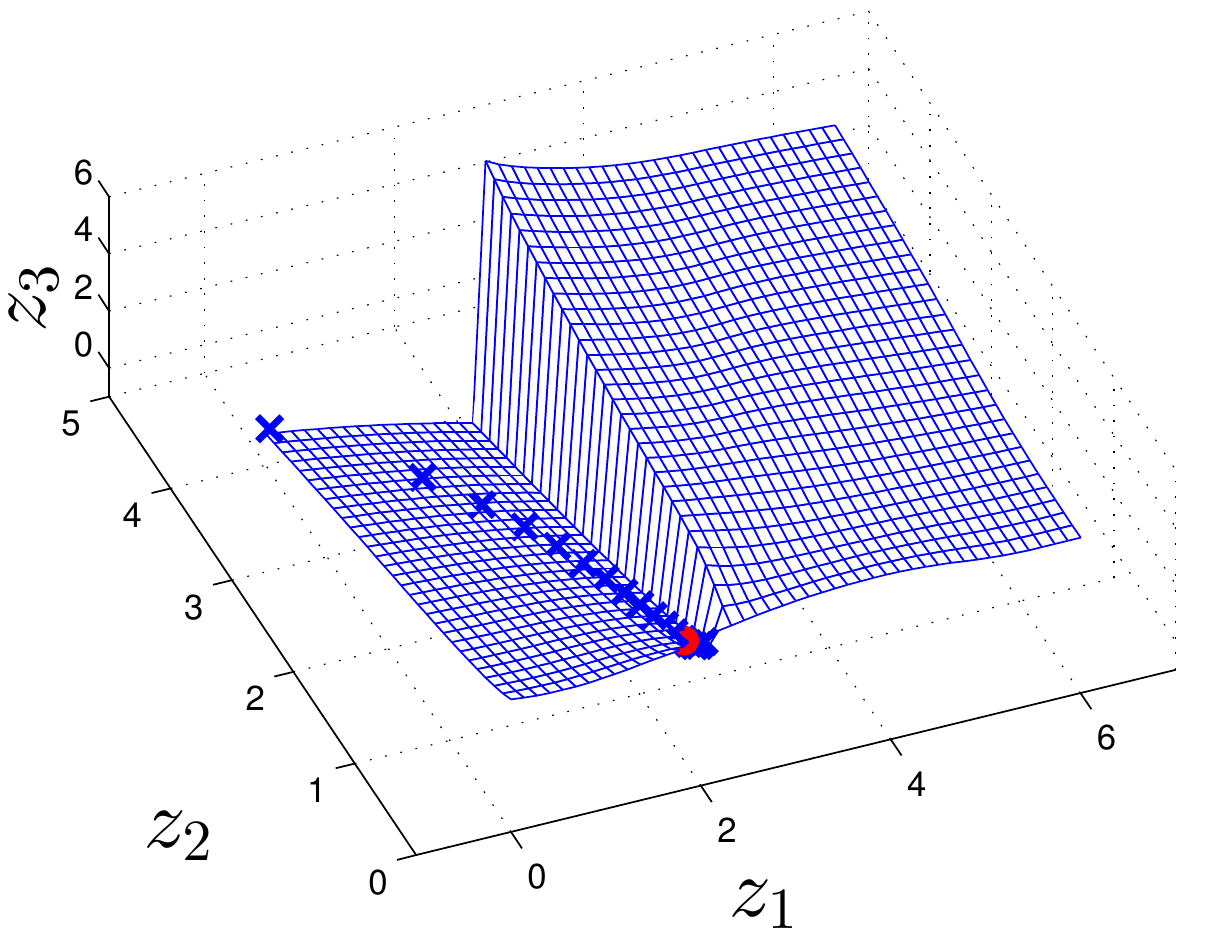}
 \caption
 {\label{f:ice2d}Visualization of numerical solutions of the optimization problem
 (\ref{eq:iop}) with $t_*=t_\f$, $t_\f-t_0=\SI{e-7}{\second}$ to reduce the
 simplified hydrogen combustion model with the reaction progress variables $z_1$
 and $z_2$. The solution points $z^*(t_*)$ with
 $z_3=z_{\species{H}_2\species{O}}$ for different values of
 $(z_1^{t_*},z_2^{t_*})$ are shown as a mesh. The results shown in
 Figure~\ref{f:result1} are again plotted as x marks as well as the equilibrium
 is shown as full dot.}
\end{figure}

Local solutions of the optimization problem (\ref{eq:iop}) approximate different
two-di\-men\-sional SIM for different values of the reaction progress variables.
The two two-dimensional SIM come close to each other. Such occurrences can lead
to severe numerical problems in the optimization algorithms as well as in the
continuation algorithms as there are regions where the KKT matrix might be
singular at least in the range of machine precision.

\subsubsection{Optimization landscapes}
We want to further illustrate this situation via optimization landscapes,
i.e.\ graphical representations of the objective function versus the (free)
optimization variables.

\paragraph{One reaction progress variable}
In Figure~\ref{f:ice-1d-log-land}, the value of
(the objective function of the optimization problem (\ref{eq:local_op}))
$\Phi=\| J_{S}(z)\;S(z) \|_2^2$ is plotted versus the two degrees of
freedom in (\ref{eq:local_op}) represented by $z_1$ and $z_2$
for one exemplary value of the one reaction progress variable
$z_3^{t_*}=z_{\species{H}_2\species{O}}^{t_*}=\SI{1}{\mole\per\kilogram}$ for the
simplified hydrogen combustion model.
\begin{figure}[htbp]
 \centering
 \includegraphics[width=0.75\textwidth]{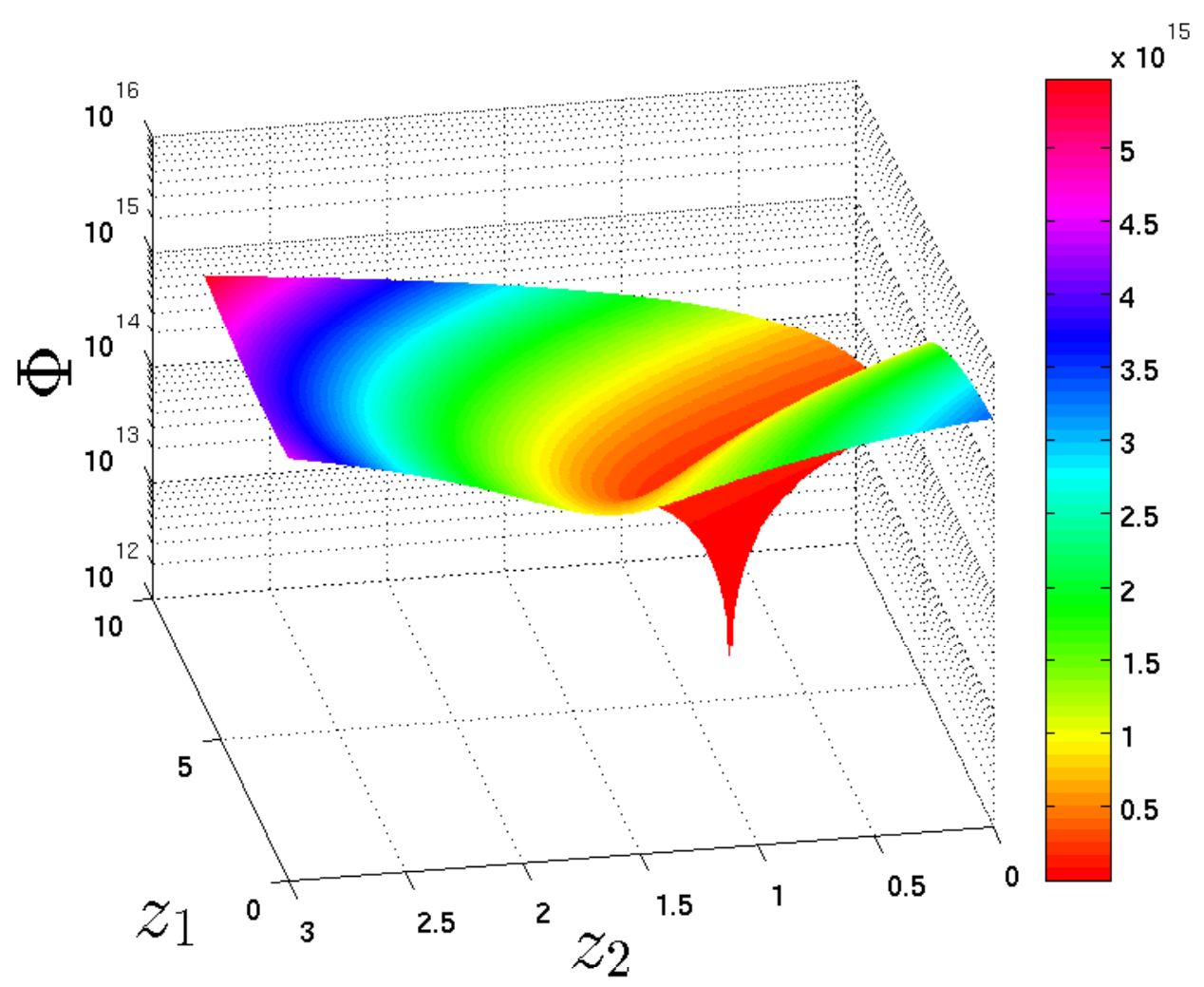}
 \caption
 {\label{f:ice-1d-log-land}Visualization of $\Phi=\| J_{S}(z)\;S(z) \|_2^2$
 in dependence of $z_1$ and $z_2$ to illustrate the solution of the optimization problem (\ref{eq:local_op})
 with $z_3^{t_*}=z_{\species{H}_2\species{O}}^{t_*}=\SI{1}{\mole\per\kilogram}$
 for the reduction of the simplified hydrogen combustion model.}
\end{figure}

The $\Phi$-axis has logarithmic scale. A single local solution of the optimization
problem (\ref{eq:local_op}) for the reduction of the simplified hydrogen
combustion model can be seen.

\paragraph{Two reaction progress variables}
To compute an optimization landscape in case of two
reaction progress variables, we fix $z_2^{t_*}=z_{\species{O}}^{t_*}=\SI{0.3}{\mole\per\kilogram}$.
We regard the value of $\Phi=\| J_{S}(z)\;S(z) \|_2^2$
(the objective function of the optimization problem (\ref{eq:local_op}))
in dependence of $z_3$ for fixed $z_1$.

The results are shown in Figure~\ref{f:ice-2d-log-land}. It can be seen that
there are two distinct local minima of $\Phi=\| J_{S}(z)\;S(z) \|_2^2$ for a
fixed value of e.g.\ $z_1=\SI{2}{\mole\per\kilogram}$: There is no unique local
solution of the optimization problem (\ref{eq:local_op}) for the reduction of
the simplified hydrogen combustion model with the reaction progress variables 
$z_2^{t_*}=z_{\species{O}}^{t_*}=\SI{0.3}{\mole\per\kilogram}$ and
$z_1^{t_*}=z_{\species{H}_2}^{t_*}=\SI{2}{\mole\per\kilogram}$.
One solution is near the value
$z_3^{\prime}=z_{\species{H}_2\species{O}}^{\prime}=\SI{3.1020}{\mole\per\kilogram}$
and the other one near $z_3^{\prime\prime}=\SI{5.2977}{\mole\per\kilogram}$.
\begin{figure}[htbp]
 \centering
 \includegraphics[width=0.75\textwidth]{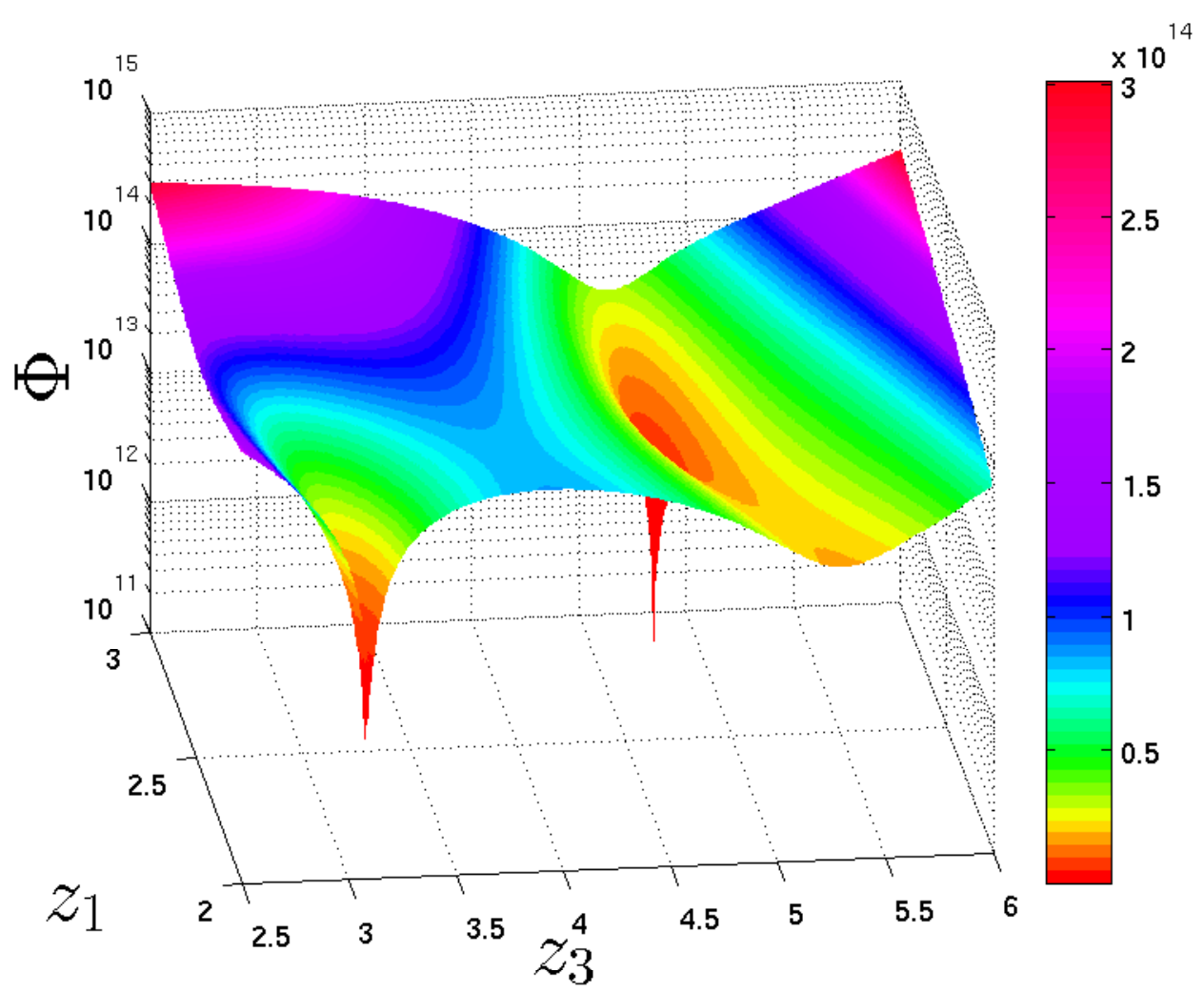}
 \caption[Six species $\species{H}_2$ comb.: opt.\ landscape, 2-dim.\ SIM, side view]
 {\label{f:ice-2d-log-land}Visualization of $\Phi=\| J_{S}(z)\;S(z) \|_2^2$
 for the simplified hydrogen combustion model
 in dependence of $z_1$ and $z_3$ for a fixed value of
 $z_2=\SI{0.3}{\mole\per\kilogram}$. The scale for $\Phi$ is logarithmic.
 Consider the value of $\Phi$ in dependence of $z_3$ for a fixed value of
 $z_1$, e.g.\ for $z_1=\SI{2}{\mole\per\kilogram}$.}
\end{figure}

In this case the predictor corrector scheme helps to follow the desired local
optimal solution in dependence of the reaction progress variables and not to
switch to another curve $c$ of local solutions. This is only possible as long as
the KKT matrix is not ``too ill-conditioned''.

\subsubsection{Performance test}
We use the specific moles of $\species{H_2O}$ and $\species{H_2}$ for the
parametrization of a two-dimensional SIM approximation in a performance test.

We consider a test situation of a two-dimensional grid of 108 points
$(z_{\species{H_2O}}^{t_*}$, $z_{\species{H_2}}^{t_*})\in[0.001,0.5,1,1.5,\dots,5,5.5]\times[0.001,0.5,1,\dots,3.5,4]$,
where points which
violate mass conservation in combination with the positivity constraints are
ignored such that 80 points remain.

Solutions of the optimization problem (\ref{eq:local_op}) computed with the
GGN method with Euler prediction for initialization of the algorithm to solve
neighboring problems are shown in Figure~\ref{f:result3} for
illustration.
\begin{figure}[htb]
 \centering
 \includegraphics[width=0.75\textwidth]{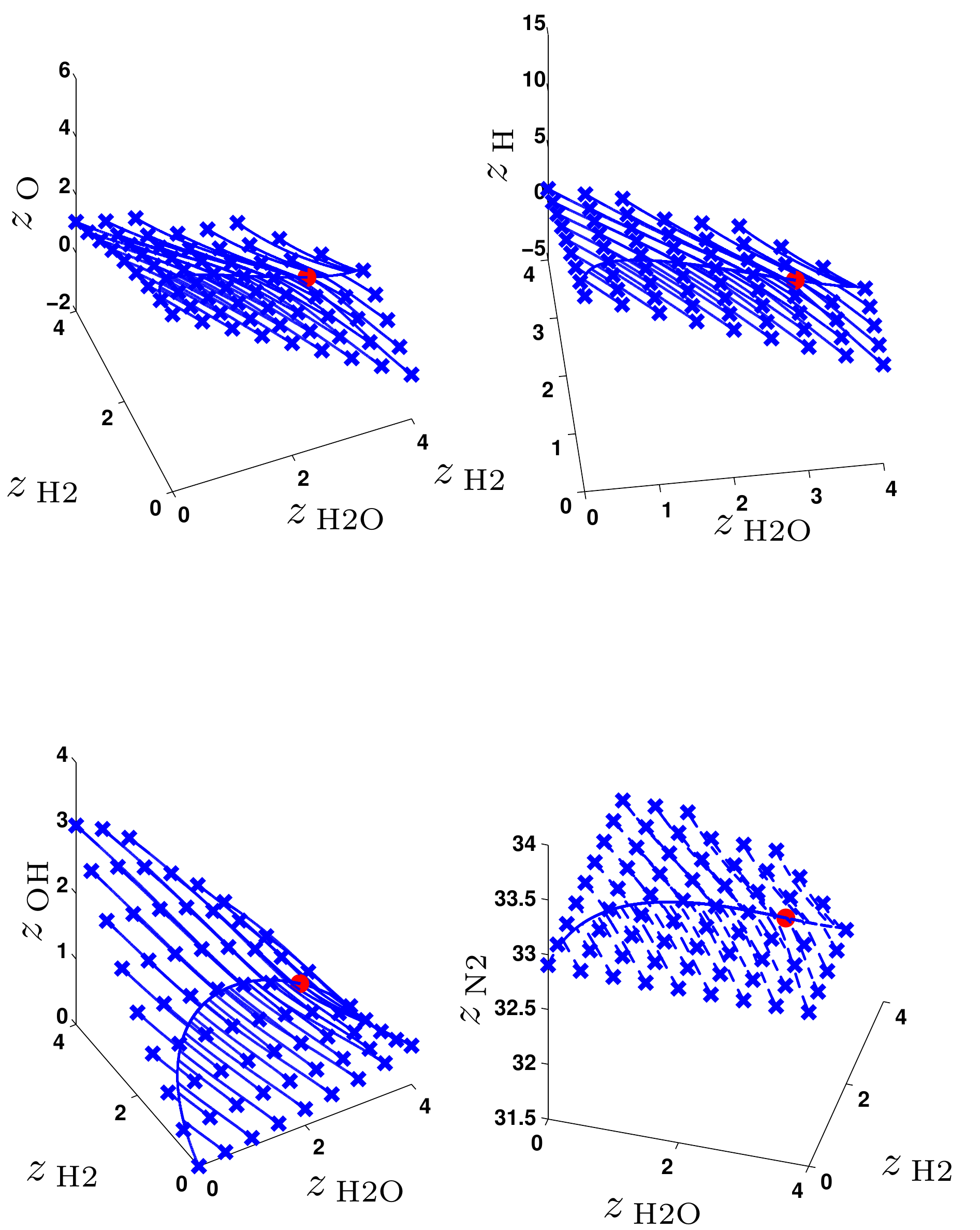}
 \caption[Six species $\species{H}_2$ comb.: optimization problem (\ref{eq:local_op}), 2-dim.\ SIM]
 {\label{f:result3}Illustration of an approximation of a two-dimensional SIM
 in the phase space of the simplified model for hydrogen combustion. The
 optimization problem (\ref{eq:local_op}) is solved with the GGN method.
 The absolute tolerance for all $d_i$ is $\SI{e-10}{}$. The relative
 tolerance is $\SI{e-9}{}$. The values of the optimization
 variables at the solution are plotted versus the given values of the
 reaction progress variables shown as x marks. The trajectories started with
 the solution $z^*(t_*)$ as initial values
 converge toward the equilibrium shown as full dot.}
\end{figure}

In Table~\ref{t:comparison}, we compare the performance of the algorithm using
the different implemented solution methods.
We use the generalized Gauss--Newton method as described in Section~\ref{s:GGN}
for the solution of (\ref{eq:local_op}). We solve the same problem with the
interior point algorithm \texttt{Ipopt} \cite{Waechter2006}.
Third, we apply a shooting approach for the
semi-infinite optimization problem (\ref{eq:gen_op}) with $t_*=t_\f$. The NLP
problem (after solving the ODE with the BDF integrator developed by D.~Skanda in
\cite{Skanda2012}) is solved with
\texttt{Ipopt} \cite{Waechter2006}. As a fourth
alternative, we use a simple Gauss--Radau collocation with linear
polynomials (backward Euler) for (\ref{eq:gen_op})
with $t_*=t_\f$. The resulting high-dimensional NLP problem is also solved with
\texttt{Ipopt} \cite{Waechter2006} in the latter case.
For the optimization problem (\ref{eq:iop}), an integration horizon
of $t_\f-t_0=\SI{e-8}{s}$ is used.
\begin{table}[htbp] \centering
 \caption
  {\label{t:comparison}Comparison of the performance of the various algorithms
  for the reduction of the simplified model for hydrogen combustion with two
  reaction progress variables and full step predictor corrector method for the
  initialization of neighboring problems.}
 \begin{tabular}{@{}llrrrr@{}}
 \toprule
  Method & Prediction & \# Iter.\ w/o fail & Time & Fail & Time w/o fail \\
  \cmidrule(r){1-1} \cmidrule(lr){2-2} \cmidrule(lr){3-3} \cmidrule(lr){4-4} \cmidrule(lr){5-5} \cmidrule(l){6-6}
  GGN         & Constant & 806 & \SI{  0.43}{s} & 11 & \SI{0.18}{s} \\
              & Euler    & 731 & \SI{  0.35}{s} & 11 & \SI{0.17}{s} \\
  IP for (\ref{eq:local_op}) & Constant & 670 & \SI{  1.09}{s} & 0  & \SI{1.09}{s} \\
              & Euler    & 591 & \SI{  0.85}{s} & 0  & \SI{0.85}{s} \\
  Shooting    & Constant & 335 & \SI{ 44.15}{s} & 0  & \SI{44.15}{s} \\
              & Euler    & 258 & \SI{ 39.36}{s} & 0  & \SI{39.36}{s} \\
  Collocation & Constant & 315 & \SI{ 95.61}{s} & 0  & \SI{ 95.61}{s} \\
              & Euler    & 234 & \SI{ 51.13}{s} & 0  & \SI{ 51.13}{s} \\
 \bottomrule
\end{tabular}
\end{table}

An initialization of neighboring problems is done without step size control,
as we evaluate the benefit of the Euler prediction here. The computations are
done with the same computer configuration as in the example for a warm start
with step size control in Section~\ref{ss:res_1d}.

It can be seen
that the gain in the computation time achieved by the Euler prediction
in case of the collocation method, where we use 100
collocation points, is the largest with
about 46.5\%. This is reasonable because of the strong dependence of the
collocation method on the initialization on all collocation points in the time
interval $[t_0,t_\f]$.

The benefit of the Euler prediction is about 10.8\% if the shooting approach is
used to solve (\ref{eq:iop}). In case of the GGN method, the benefit is only
5.6\% if we do not regard failures. The eleven failures only
occur in the region, where $z_{\species{H_2O}}^{t_*}$ is larger than the
equilibrium value $z_{\species{H_2O}}^{\textrm{eq}}$. These can
not be overcome neither with the step size control for the continuation
method nor with a larger tolerance for convergence in the GGN method. It can
be seen that the algorithm for the solution of (\ref{eq:local_op}) with the GGN
method is much faster than the algorithm for the solution of
(\ref{eq:iop}). The computation of one approximation of a point on the SIM
(without regarding failures) takes about
$\SI{0.17}{\second}/(80-11) \approx \SI{2.5}{\milli\second}$. This is pretty
fast and might be used in an online (\emph{in situ}) SIM computation during CFD
simulations.

\section{Conclusion} \label{s:conclusion}
In this article, we present a strategy for an efficient solution of parametric
optimization problems that arise for a method for model reduction that is
raised in \cite{Lebiedz2004c,Lebiedz2010,Reinhardt2008}.

Two methods are tested to solve the semi-infinite optimization problem
(\ref{eq:iop}) for model reduction: collocation and shooting approaches.
The resulting nonlinear programming problem is solved with a state-of-the-art
open source interior point algorithm \cite{Waechter2006}. It turns out that all
variants for the solution of the semi-infinite optimization problem are too slow
for an \emph{in situ} application of the model reduction method in e.g.\ a
computational fluid dynamics simulation.

A finite optimization problem in form of a constrained nonlinear least squares
problem can be formulated instead. This can be solved efficiently
with a generalized Gauss--Newton method. A filter approach is used for
globalization of convergence. Second order correction iterations prevent the
Maratos effect. The problem that has to be solved in the feasibility restoration
phase of the filter algorithm can also be formulated as a least squares problem
such that matrix factorizations can be reused.

Parameter sensitivities of the optimization problem are used in a homotopy
method for the solution of neighboring problems. The reaction progress variables
for parametrization of the slow manifold
are considered as parameters in the optimization problem. The problem has to be
solved several times for different values of the parameters.
An Euler prediction of the solution of neighboring problems is employed based on
the parameter sensitivities, which are computed to obtain an approximation of
the tangent space of the slow manifold.
This linear prediction can be used directly within a certain tolerance as an
approximation of a point on the slow manifold. It can also be used in
combination with a step size strategy.
The step size is computed according to \cite{Heijer1981} in dependence of the
contraction of the corrector method -- the solution method for the nonlinear
programming problem -- and the iterations needed to solve a previous
optimization problem.

We study the presented methods for a simplified model for hydrogen combustion.
We find that a solution of the optimization problem for approximation of
points on a slow manifold can be computed in short time for the presented
example. Furthermore it can be seen that there might exist several distinct
local solutions of the optimization problem. In such cases, the predictor
corrector scheme helps to follow one local minimum in dependence of the chosen
parametrization of the slow manifold.

Such a solution strategy could also be used for variations in the model
parameters as e.g.\ the mixture fraction, the internal energy, and others.


\end{document}